\documentclass{article}

\usepackage{lineno}
\usepackage{amssymb}
\usepackage{amsmath}
\usepackage{latexsym}
\usepackage{graphicx}
\usepackage{array,epsf,epsfig}
\usepackage{mathdots}
\usepackage{booktabs}
\usepackage{float}
\usepackage{multicol}
\usepackage{multirow}
\usepackage{bbm}
\usepackage{color,soul}
\usepackage{cite}
\usepackage{hyperref}
\usepackage{caption}

\newtheorem{lemma}{Lemma}
\newtheorem{theorem}{Theorem}
\newtheorem{proposition}{Proposition}

\newtheorem{remark}{Remark}
\newtheorem{example}{Example}
\newenvironment{proof}{\noindent{\bf Proof:}}{\hfill\fbox{}\vspace*{1mm}}

\newcommand{\mF}{\mathcal{F}}

\title{Upwind-and-shifted numerical scheme for fractional convection equation}
\author{Lot-Kei Chou,\thanks{Email address: rickyclk16@gmail.com}\and
Wan-Na Deng,\thanks{Email address: mc05564@um.edu.mo}\and
Yuan-Yuan Huang,\thanks{Email address: yc07487@um.edu.mo}\and
Siu-Long Lei\thanks{Email address: sllei@um.edu.mo (Corresponding author)}\\
Department of Mathematics, University of Macau, Macao, China} 

\begin{document}

\maketitle

\begin{abstract}
Fundamental solution of a space fractional convection equation of order $\alpha$ is the probability density function of L\'{e}vy flights with long-tailed $\alpha$-stable jump length distribution. By studying an upwind second-order implicit finite difference scheme for the equation with $\alpha\in(0,1)$, an upwind-and-shifted scheme with order $3-\alpha$ is obtained in this paper, and the scheme is shown to be unconditionally stable for a wide range of $\alpha$. Numerical examples, including simulations on a probability density function, are presented showing the effectiveness of the numerical schemes.
\end{abstract}


\section{Introduction}
For several decades, numerous fractional order models have been suggested to be describing anomalous phenomena appropriately~\cite{SunZhang,Sokolov,Podlubny}, and various numerical methods for solving fractional differential equations (FDEs) have been proposed~\cite{LiZeng,Liu,SunGao}. Often, the fractional orders $\alpha$ in the FDEs are restricted within the interval $(0,2]$, in both theoretical aspect~\cite{Zaslavsky,Feller} and empirical aspect~\cite{Imag,Bio,Hydro,Turb}. FDEs with temporal order $\alpha\in(0,1)$ or spatial order $\alpha\in(1,2)$ have been applied for models such as anomalous diffusion~\cite{Phy}, solute transport in porous media~\cite{Porous}, and magnetic resonance~\cite{MR}, with related numerical methods proposed in~\cite{TFDE2,TFDE1,SFDE2,SFDE1,TSFDE2,TSFDE1}. While FDEs with spatial order $\alpha\in(0,1)$ have been considered for models such as continuous-time random walk~\cite{Meerschaert}, finance~\cite{Finance}, and super-diffusion~\cite{Super}, numerical methods are to be further explored~\cite{FE2,SDO,FE1}.

Recently, due to the important applications of L\'{e}vy flight \cite{Padash}, Jesus and Sousa~\cite{Sousa} demonstrated that the jump distribution function of a class of L\'{e}vy stable processes~\cite{Foge} satisfies a fractional convection equation with order $\alpha\in(0,1)\cup(1,2)$. Then they developed central and upwind approximations to R-L fractional derivatives of such order $\alpha$, and the associated implicit schemes for solving fractional convection equation are shown to be consistent and unconditionally stable in the sense of von Neumann analysis. Their upwind second-order scheme can also effectively simulate the L\'{e}vy-Smirnov distribution of L\'{e}vy flights. Meanwhile, in their proof of the unconditional stability, the key point is to prove that the related generating function, which is an infinite series of trigonometric function, is non-negative, see Lemma \ref{L2}.

In this paper, we first consider the upwind second-order scheme in \cite{Sousa}, and the above mentioned generating function is to be studied based on the Lerch transcendent~\cite{Erdelyi}. As an alternative to the von Neumann sense, the unconditional stability and convergence of the scheme are hence to be verified in the sense that the error in each time level is bounded by the initial error. Benefiting from the intended study about the generating function, a form of upwind-and-shifted approximation is introduced, expecting to obtain an order $3-\alpha$ convergence and the unconditional stability. Indeed, this approach can be viewed as an application of the idea in~\cite{WSGD} to a generating function different from $(1-e^{i\theta})^{\alpha}$. With an alternate form of the generating function, the related numerical stability can be determined more easily. On the other hand, due to the long-tail property of probability density function of L\'{e}vy flights, a large spatial domain for the fractional convection equation is required so that the discrete linear system is very large in order to obtain accurate simulations. To solve the linear system efficiently, iterative methods such as the GMRES method can be utilized due to the Toeplitz structure of the coefficient matrix.

In Section~\ref{sec2}, a second-order approximation of fractional derivatives is reviewed, followed by discussions about upwind-and-shifted approximations. The associated numerical schemes are provided in Section~\ref{sec3}, with the stability and convergence obtained in Section~\ref{sec4}. Numerical results are given in Section~\ref{sec5}.

\section{Approximations of fractional derivatives}\label{sec2}
In this work, we consider the space fractional convection equation with order $\alpha\in(0,1)$:
\begin{equation}\label{fde}
\frac{\partial u(x, t)}{\partial t}=-D\left(\frac{1}{2}(1+\beta) \frac{\partial^{\alpha} u}{\partial x^{\alpha}}(x, t)+\frac{1}{2}(1-\beta) \frac{\partial^{\alpha} u}{\partial(-x)^{\alpha}}(x, t)\right)+q(x, t),
\end{equation}
where $x \in \mathbb{R},~t>0$, $D$ and $\beta$ are constants with $D>0$, $-1\leq\beta\leq 1$, $q(x,t)$ is the source term, and
\begin{align*}
&\frac{\partial^{\alpha} u}{\partial x^{\alpha}}(x, t)=\frac{1}{\Gamma(1-\alpha)} \frac{\partial}{\partial x} \int_{-\infty}^{x} u(s, t)(x-s)^{-\alpha} d s,\\
&\frac{\partial^{\alpha} u}{\partial(-x)^{\alpha}}(x, t)=\frac{-1}{\Gamma(1-\alpha)} \frac{\partial}{\partial x} \int_{x}^{\infty} u(s, t)(s-x)^{-\alpha} d s.
\end{align*}

Denote 
\begin{align*}
V^{\gamma}=\left\{u~\Bigg|~\dfrac{d^{s}u(x)}{d(\pm x)^{s}}\in C(\mathbb{R})\cap L^{1}(\mathbb{R}),~\forall s\in[0,\gamma],\textrm{ and }\right.\\
\left.\lim\limits_{x\to\infty}\dfrac{d^{s}u(x)}{d(\pm x)^{s}}=0,~\forall s\in[0,\gamma-1]\right\}.
\end{align*}

Also, denote $a_{0}=1,~a_{k}=(k+1)^{2-\alpha}-2k^{2-\alpha}+(k-1)^{2-\alpha},~k\geq 1$, and
\begin{align*}
I_{1}(\theta,\alpha)&=a_{0}+\sum\limits_{k=1}^{\infty}(a_{k}-a_{k-1})e^{ik\theta}:=\sum\limits_{k=0}^{\infty}b_{k}e^{ik\theta},\\
I_{2}(\theta,\alpha)&=a_{0}+(a_{1}-2a_{0})e^{i1\theta}+\sum\limits_{k=2}^{\infty}(a_{k}-2a_{k-1}+a_{k-2})e^{ik\theta}:=\sum\limits_{k=0}^{\infty}c_{k}e^{ik\theta},\\
I_{3}(\theta,\alpha)&=a_{0}+(a_{1}-3a_{0})e^{i1\theta}+(a_{2}-3a_{1}+3a_{0})e^{i2\theta}\\
&\qquad+\sum\limits_{k=3}^{\infty}(a_{k}-3a_{k-1}+3a_{k-2}-a_{k-3})e^{ik\theta}:=\sum\limits_{k=0}^{\infty}d_{k}e^{ik\theta}.
\end{align*}

Then the studies about the approximations of the fractional derivatives $\partial^{\alpha}u/\partial(\pm x)^{\alpha}$ are presented in the following context.

\subsection{Upwind approximation}
In~\cite{Sousa}, an upwind and the corresponding downwind second-order approximations are proposed as
\begin{align}\label{2ndapp}
&A_{\alpha,1}^l u(x,t):=\dfrac{1}{2\Gamma(3-\alpha)}\sum\limits_{k=0}^{\infty}(2b_{k}+c_{k})f(x-kh)\approx\dfrac{d^{\alpha}f(x)}{dx^{\alpha}},\nonumber\\
&A_{\alpha,1}^r u(x,t):=\dfrac{1}{2\Gamma(3-\alpha)}\sum\limits_{k=0}^{\infty}(2b_{k}+c_{k})f(x+kh)\approx\dfrac{d^{\alpha}f(x)}{d(-x)^{\alpha}},
\end{align}
with the sequence $\{2b_{k}+c_{k}\}_{k=0}^{\infty}$ satisfies the following property.

\begin{lemma} \label{L2}
For all $\alpha\in(0,1)$, we have 
\[\sum\limits_{k=0}^{\infty}(2b_{k}+c_{k}) \cos (k \theta) \geq 0, \mbox{ for all } \theta\in\mathbb{R}.\]
\end{lemma}

Here, we present a proof of this property based on the Lerch transcendent. Indeed, we need the following two propositions just before the proof.

\begin{proposition}\label{lt}
The Lerch transcendent $\Phi(z,s,\nu)=\sum_{k=0}^{\infty}(k+\nu)^{-s}z^{k}$ with $z\in\mathbb{C}\setminus[1,\infty)$ has the following analytic continuation with respect to $s$ (\cite{Erdelyi}, 1.11).
\begin{align*}
\begin{cases}
\Phi\vert_{S_{1}}(z,s,\nu)=z^{-\nu}\Gamma(1-s)\sum\limits_{n=-\infty}^{\infty}(-\log z+2n\pi i)^{s-1}e^{2n\pi i\nu},\\
\textrm{for }\lvert\arg(-\log z+2n\pi i)\rvert\leq\pi,~0<\nu\leq 1,~s\in S_{1}=\{s\in\mathbb{C}\big\vert\textrm{Re}~s<0\};\\
\Phi\vert_{S_{2}}(z,s,\nu)=\dfrac{\Gamma(1-s)}{z^{\nu}}\left(-\log z\right)^{s-1}+z^{-\nu}\sum\limits_{r=0}^{\infty}\zeta(s-r,\nu)\dfrac{(\log z)^{r}}{r!},\\
\textrm{for }\lvert\log z\rvert<2\pi,~\nu\neq 0,-1,-2,\ldots,~s\in S_{2}=\{s\in\mathbb{C}\big\vert s\neq 1,2,3,\ldots\}.
\end{cases}
\end{align*}
\end{proposition}

Denote $E_{s}=\{z\in\mathbb{C}\big\vert\textrm{Re}~z>s,~z\neq 3,4,5,\ldots\},~\tilde{E}_{s}=\{z\in\mathbb{C}\big\vert\textrm{Re}~z<s\}$, and the expressions
\begin{align*}
&J_{1}(\theta,\alpha)=8ie^{\frac{i\theta}{2}}\sin^{3}\dfrac{\theta}{2}e^{i\theta}\Phi\vert_{S_{1}}(e^{i\theta},\alpha-2,1),\\
&J_{2}(\theta,\alpha)=16e^{i\theta}\sin^{4}\dfrac{\theta}{2}e^{i\theta}\Phi\vert_{S_{1}}(e^{i\theta},\alpha-2,1),\\
&K_{1}(\theta,\alpha)=8ie^{\frac{i\theta}{2}}\sin^{3}\dfrac{\theta}{2}e^{i\theta}\Phi\vert_{S_{2}}(e^{i\theta},\alpha-2,1),\\
&K_{2}(\theta,\alpha)=16e^{i\theta}\sin^{4}\dfrac{\theta}{2}e^{i\theta}\Phi\vert_{S_{2}}(e^{i\theta},\alpha-2,1).
\end{align*}

With Proposition~\ref{lt}, we have the following proposition.

\begin{proposition}\label{ac}
For each fixed $\theta\in(0,\pi]$,
\begin{enumerate}
\item[(i)]
$I_{1}(\theta,\alpha)$ and $I_{2}(\theta,\alpha)$ are analytic on $E_{0}$ with respect to $\alpha$.
\item[(ii)]
For $\alpha\in(0,1)$, 
\begin{enumerate}
\abovedisplayskip=-\baselineskip
\item
\begin{align*}
I_{1}(\theta,\alpha)&=J_{1}(\theta,\alpha)\\
&=8ie^{\frac{i\theta}{2}}\sin^{3}\dfrac{\theta}{2}\Gamma(3-\alpha)\times\\
&\quad\sum\limits_{n=0}^{\infty}\left((2n\pi+\theta)^{\alpha-3}e^{\frac{i(3-\alpha)\pi}{2}}+(2(n+1)\pi-\theta)^{\alpha-3}e^{-\frac{i(3-\alpha)\pi}{2}}\right);
\end{align*}
\item
\begin{align*}
I_{2}(\theta,\alpha)&=J_{2}(\theta,\alpha)\\
&=16e^{i\theta}\sin^{4}\dfrac{\theta}{2}\Gamma(3-\alpha)\times\\
&\quad\sum\limits_{n=0}^{\infty}\left((2n\pi+\theta)^{\alpha-3}e^{\frac{i(3-\alpha)\pi}{2}}+(2(n+1)\pi-\theta)^{\alpha-3}e^{-\frac{i(3-\alpha)\pi}{2}}\right);
\end{align*}
\item
\begin{align*}
I_{1}(\theta,\alpha)&=K_{1}(\theta,\alpha)\\
&=8ie^{\frac{i\theta}{2}}\sin^{3}\dfrac{\theta}{2}\left(\Gamma(3-\alpha)\left(-i\theta\right)^{\alpha-3}+\sum\limits_{r=0}^{\infty}\zeta(\alpha-2-r)\dfrac{(i\theta)^{r}}{r!}\right);
\end{align*}
\item
\begin{align*}
I_{2}(\theta,\alpha)&=K_{2}(\theta,\alpha)\\
&=16e^{i\theta}\sin^{4}\dfrac{\theta}{2}\left(\Gamma(3-\alpha)\left(-i\theta\right)^{\alpha-3}+\sum\limits_{r=0}^{\infty}\zeta(\alpha-2-r)\dfrac{(i\theta)^{r}}{r!}\right).
\end{align*}
\end{enumerate}
\end{enumerate}
\end{proposition}

\begin{proof}
For part (i), denote
\begin{align*}
I_{1}(\theta,\alpha)=1+&(2^{2-\alpha}-3)e^{i\theta}+(3^{2-\alpha}-3\cdot 2^{2-\alpha}+3)e^{i2\theta}+\sum\limits_{k=3}^{\infty}b_{k}e^{ik\theta}\\
:=1+&(2^{2-\alpha}-3)e^{i\theta}+(3^{2-\alpha}-3\cdot 2^{2-\alpha}+3)e^{i2\theta}+\sum\limits_{k=3}^{\infty}f_{k}(\alpha)e^{ik\theta}.
\end{align*}

Also, denote $\sigma=\textrm{Re}~\alpha$ and $\tau=\textrm{Im}~\alpha$. Then for $k\geq 3$, by mean value theorem for divided difference, there is some $\xi_{k},\eta_{k}\in(k-2,k+1)$ such that
\begin{align*}
&\left\vert\textrm{Re}(f_{k}(\alpha))\right\vert\\
&=\left\vert(k+1)^{2-\sigma}\cos(\tau\log(k+1))-3k^{2-\sigma}\cos(\tau\log k)\right.\\
&\qquad\left.+3(k-1)^{2-\sigma}\cos(\tau\log(k-1))-(k-2)^{2-\sigma}\cos(\tau\log(k-2))\right\vert\\
&=\left\vert\dfrac{d^{3}}{dx^{3}}(x^{2-\sigma}\cos(\tau\log x))\bigg\vert_{x=\xi_{k}}\right\vert\\
&\leq C_{1}(\alpha)\xi_{k}^{-1-\sigma},
\end{align*}\
and
\begin{align*}
&\left\vert\textrm{Im}(f_{k}(\alpha))\right\vert\\
&=\left\vert(k+1)^{2-\sigma}\sin(\tau\log(k+1))-3k^{2-\sigma}\sin(\tau\log k)\right.\\
&\qquad\left.+3(k-1)^{2-\sigma}\sin(\tau\log(k-1))-(k-2)^{2-\sigma}\sin(\tau\log(k-2))\right\vert\\
&=\left\vert\dfrac{d^{3}}{dx^{3}}(x^{2-\sigma}\sin(\tau\log x))\bigg\vert_{x=\eta_{k}}\right\vert\\
&\leq C_{2}(\alpha)\eta_{k}^{-1-\sigma},
\end{align*}
for some continuous functions $C_{1},~C_{2}$ on $\mathbb{C}$.

Then for $\sigma>0$,
\begin{equation*}
\vert f_{k}(\alpha)e^{i\theta}\vert\leq(C_{1}(\alpha)+C_{2}(\alpha))(k-2)^{-1-\sigma}.
\end{equation*}

As $C_{1}+C_{2}$ is also continuous on $\mathbb{C}$, and $\sum_{k=1}^{\infty}k^{-s}$ converges for $s>1$, we have for each fixed $\theta$, the series $I_{1}(\theta,\alpha)$ converges uniformly on any compact subset of $\{\alpha\in\mathbb{C}\big\vert\textrm{Re}~\alpha>0\}$. In addition, each term in $I_{1}(\theta,\alpha)$ is analytic for $\alpha\in\mathbb{C}$. Therefore $I_{1}(\theta,\alpha)$ is analytic for $\textrm{Re}~\alpha>0$, in particular for $\alpha\in E_{0}$. Similarly, $I_{2}(\theta,\alpha)$ is analytic for $\alpha\in E_{0}$.

For part (ii), we first point out that,
\begin{itemize}
\item
from $I_{1}(\theta,\alpha)$ and $K_{1}(\theta,\alpha)$ being analytic for $\alpha\in E_{0}$, and $I_{1}(\theta,\alpha)=K_{1}(\theta,\alpha)$ for $\alpha\in E_{3}$ by rearranging $I_{1}(\theta,\alpha)$ as $(1-e^{i\theta})^{3}\Phi(e^{i\theta},\alpha-2,1)$ without affecting convergence, we have $I_{1}(\theta,\alpha)=K_{1}(\theta,\alpha)$ for $\alpha\in E_{0}$ (\cite{Lang}, Chp. III, Theorem 1.2);
\item
$J_{1}(\theta,\alpha)$ and $K_{1}(\theta,\alpha)$ are analytic and $J_{1}(\theta,\alpha)=K_{1}(\theta,\alpha)$ for $\alpha\in\tilde{E}_{2}$.
\end{itemize}
From these two points, we have $I_{1}(\theta,\alpha)=J_{1}(\theta,\alpha)=K_{1}(\theta,\alpha)$ for $\alpha\in E_{0}\cap\tilde{E}_{2}\supset(0,1)$. Similarly, we have $I_{2}(\theta,\alpha)=J_{2}(\theta,\alpha)=K_{2}(\theta,\alpha)$ for $\alpha\in E_{0}\cap\tilde{E}_{2}\supset(0,1)$.\vspace{5pt}
\end{proof}

With Proposition~\ref{ac}, we are ready to prove Lemma~\ref{L2}.

\begin{proof}
Denote $g_{1}(\theta,\alpha)=\sum_{k=0}^{\infty}(2b_{k}+c_{k})e^{ik\theta}$, for fixed $\theta\in(0,\pi]$, we have 
\begin{equation*}
g_{1}\vert_{(0,1)}(\theta,\alpha)=2I_{1}\vert_{(0,1)}(\theta,\alpha)+I_{2}\vert_{(0,1)}(\theta,\alpha).
\end{equation*}

Then from $I_{1}\vert_{(0,1)}(\theta,\alpha)=J_{1}\vert_{(0,1)}(\theta,\alpha)$ and $I_{2}\vert_{(0,1)}(\theta,\alpha)=J_{2}\vert_{(0,1)}(\theta,\alpha)$ by Proposition~\ref{ac}, and with elementary simplifications, we have
\begin{align}\label{g1}
&2\sum\limits_{k=0}^{\infty}(2b_{k}+c_{k})\cos(k\theta)\nonumber\\
&=g_{1}\vert_{(0,1)}(\theta,\alpha)+\bar{g}_{1}\vert_{(0,1)}(\theta,\alpha)\nonumber\\
&=2(J_{1}\vert_{(0,1)}(\theta,\alpha)+\bar{J_{1}}\vert_{(0,1)}(\theta,\alpha))+(J_{2}\vert_{(0,1)}(\theta,\alpha)+\bar{J_{2}}\vert_{(0,1)}(\theta,\alpha))\nonumber\\
&=32\sin^{3}\dfrac{\theta}{2}\Gamma(3-\alpha)\Big(g_{11}(\theta)H_{1}(\theta,\alpha)+g_{12}(\theta)H_{2}(\theta,\alpha)\Big),
\end{align}
where
\begin{align*}
&g_{11}(\theta)=\cos\dfrac{\theta}{2}+\sin\dfrac{\theta}{2}\sin\theta\geq 0,\\
&g_{12}(\theta)=\sin\dfrac{\theta}{2}-\sin\dfrac{\theta}{2}\cos\theta>0,\\
&H_{1}(\theta,\alpha)=\cos\dfrac{\alpha\pi}{2}\sum\limits_{n=0}^{\infty}\left((2n\pi+\theta)^{\alpha-3}-(2(n+1)\pi-\theta)^{\alpha-3}\right)>0,\\
&H_{2}(\theta,\alpha)=\sin\dfrac{\alpha\pi}{2}\sum\limits_{n=0}^{\infty}\left((2n\pi+\theta)^{\alpha-3}+(2(n+1)\pi-\theta)^{\alpha-3}\right)>0.
\end{align*}

Hence $\sum_{k=0}^{\infty}(2b_{k}+c_{k})\cos(k\theta)\geq 0,~\forall\theta\in\mathbb{R}$, by continuity, symmetry, and periodicity.
\end{proof}

\begin{remark}
Besides inspecting the partial sum, taking $\theta\to 0^{+}$ in~\ref{g1} also reads $\sum_{k=0}^{\infty}(2b_{k}+c_{k})=0$.
\end{remark}

\begin{remark}
In particular, for $\alpha\in[0.263, 0.577]$, $b_k\leq0$, for all $k\geq 1$, and then the above inequality can be proved directly as
\begin{equation*}
\sum_{k=0}^{\infty} (2b_k+c_k)\cos(k\theta) \geq 2b_0+c_0-\sum_{k=1}^{\infty} \lvert (2b_{k}+c_{k})\rvert=\sum_{k=0}^{\infty}(2b_k+c_k)=0.
\end{equation*}
\end{remark}

Proposition~\ref{ac} also relates to the consistency of the approximations in~(\ref{2ndapp}), as described in the following proposition.

\begin{proposition}\label{cs}
If $f\in V^{\gamma}$ for some $\gamma>3+\alpha$, then for all $x\in\mathbb{R}$, 
\begin{align*}
\left\vert\dfrac{d^{\alpha} f}{d x^{\alpha}}(x)-A_{\alpha,1}^l f(x)\right\vert\leq M_{1}h^{2}\quad\textrm{and}\quad\left\vert\frac{d^{\alpha} f}{d (-x)^{\alpha}}(x)-A_{\alpha,1}^r f(x)\right\vert\leq M_{1}h^{2}
\end{align*}
for some $M_{1}>0$ independent of $x,h$.
\end{proposition}

\begin{proof}
We prove for the left-sided case here, the right-sided case is similar.

Consider $g_{1}(\theta,\alpha)=\sum_{k=0}^{\infty}(2b_{k}+c_{k})e^{ik\theta}$ as mentioned in the proof of Lemma~\ref{L2}, we first have $g_{1}\vert_{(0,1)}(0,\alpha)=2I_{1}\vert_{(0,1)}(0,\alpha)+I_{2}\vert_{(0,1)}(0,\alpha)=0$.

For $0<\vert\theta\vert<\theta_{0}$ for small $\theta_{0}$, we have $I_{1}\vert_{(0,1)}(\theta,\alpha)=K_{1}\vert_{(0,1)}(\theta,\alpha)$ and $I_{2}\vert_{(0,1)}(\theta,\alpha)=K_{2}\vert_{(0,1)}(\theta,\alpha)$ by Proposition~\ref{ac}, then
\begin{align*}
&g_{1}\vert_{(0,1)}(\theta,\alpha)\\
&=2I_{1}\vert_{(0,1)}(\theta,\alpha)+I_{2}\vert_{(0,1)}(\theta,\alpha)\\
&=2K_{1}\vert_{(0,1)}(\theta,\alpha)+K_{2}\vert_{(0,1)}(\theta,\alpha)\\
&=\left(\dfrac{\sin\frac{\theta}{2}}{\frac{\theta}{2}}\right)^{3}\left(3e^{\frac{i\theta}{2}}-e^{\frac{i3\theta}{2}}\right)\left(\Gamma(3-\alpha)\left(-i\theta\right)^{\alpha}+(-i\theta)^{3}\sum\limits_{r=0}^{\infty}\zeta(\alpha-2-r)\dfrac{(i\theta)^{r}}{r!}\right)\\
&=\left(1-\dfrac{\theta^{2}}{8}+O(\theta^{4})\right)\left(2+\dfrac{3\theta^{2}}{4}+O(\theta^{3})\right)\Big(\Gamma(3-\alpha)(-i\theta)^{\alpha}+O(\theta^{3})\Big)\\
&=2\Gamma(3-\alpha)(-i\theta)^{\alpha}+O(\theta^{2+\alpha}).
\end{align*}

On the other hand, as $g_{1}(\theta,\alpha)$ is periodic and continuous on $\mathbb{R}$ with respect to $\theta$, it is bounded, and hence for $\vert\theta\vert\geq\theta_{0}$, $\vert g_{1}(\theta,\alpha)-2\Gamma(3-\alpha)(-i\theta)^{\alpha}\vert/\vert\theta\vert^{2+\alpha}$ is also bounded. In conclusion, we have for all $\theta\in\mathbb{R}$,
\begin{align*}
\Big\vert g_{1}(\theta,\alpha)-2\Gamma(3-\alpha)(-i\theta)^{\alpha}\Big\vert\leq C_{3}\vert\theta\vert^{2+\alpha},
\end{align*}
for some $C_{3}>0$ independent of $\theta$.

From $f\in V^{\gamma}$, by Riemann-Lebesgue lemma, we have $\left\vert\mF[d^{\gamma}f/d x^{\gamma}](\xi)\right\vert=\vert\xi\vert^{\gamma}\vert\hat{f}(\xi)\vert\rightarrow 0~\textrm{as}~\vert\xi\vert\rightarrow\infty$, where $\mF[\cdot]=\hat{\cdot}$ is the Fourier transform, hence $\vert\hat{f}(\xi)\vert\leq C_{\gamma}/(1+\vert\xi\vert)^{\gamma}$ for some $C_{\gamma}>0$ independent of $\xi,x,h$. 

Combining the results, we have for $\xi\in\mathbb{R}$,
\begin{align*}
\bigg\vert\mF\left[\left(\dfrac{d^{\alpha}}{d x^{\alpha}}-A_{\alpha,1}^l\right)f\right](\xi)\bigg\vert
&=\bigg\vert(-i\xi)^{\alpha}-\dfrac{1}{2\Gamma(3-\alpha)h^{\alpha}}g_{1}(\xi h,\alpha)\bigg\vert\vert\hat{f}(\xi)\vert\\
&\leq\dfrac{1}{2\Gamma(3-\alpha)h^{\alpha}}\cdot C_{3}\vert\xi h\vert^{2+\alpha}\cdot\dfrac{C_{\gamma}}{(1+\vert\xi\vert)^{\gamma}}\\
&:=\dfrac{C_{4}\vert\xi\vert^{2+\alpha}h^{2}}{(1+\vert\xi\vert)^{\gamma}}.
\end{align*}

Since $\gamma>3+\alpha$, by Fourier inversion, we have
\begin{align*}
\left\vert\left(\dfrac{d^{\alpha}}{d x^{\alpha}}-A_{\alpha,1}^l\right)f(x)\right\vert
&\leq C_{5}h^{2}\int_{-\infty}^{\infty}\dfrac{\vert\xi\vert^{2+\alpha}}{(1+\vert\xi\vert)^{\gamma}}d\xi\\
&=M_{1}h^{2},
\end{align*}
for some $C_{5},M_{1}>0$ independent of $x,h$.
\end{proof}

In this proof, the second-order convergence of the approximations~(\ref{2ndapp}) is interpreted to be resulting from $2I_{1}+I_{2}=(2+O(\theta^{2}))(\Gamma(3-\alpha)(-i\theta)^{\alpha}+O(\theta^{3})$. As the factor $2+O(\theta^{2})$ varies with different weights of $I_{1}$ and $I_{2}$, while the factor $\Gamma(3-\alpha)(-i\theta)^{\alpha}+O(\theta^{3})$ is fixed, it is likely that $3-\alpha$ is an optimal order of approximation for this approach, and it can be achieved by introducing certain similar weights such that $O(\theta^{2})$ is replaced by $O(\theta^{3})$. At the same time, it is also desired that the weighted result has non-negative real part. Taking $I_{3}$ into account, the approach of upwind approximation is extended in the discussions below.

\subsection{Two-term upwind approximation}\label{sap1}
Similar to the proof of Proposition~\ref{cs}, it can be shown that for $\alpha\in(0,1)$, 
\[g_{2}(\theta,\alpha):=I_{1}+\dfrac{1}{2}I_{2}+\dfrac{1}{4}I_{3}=(1+O(\theta^{3}))(\Gamma(3-\alpha)(-i\theta)^{\alpha}+O(\theta^{3})),\]
so that the approximations
\begin{align}\label{ap1}
&A_{\alpha,2}^l u(x,t):=\dfrac{1}{\Gamma(3-\alpha)}\sum\limits_{k=0}^{\infty}\left(b_{k}+\dfrac{1}{2}c_{k}+\dfrac{1}{4}d_{k}\right)f(x-kh)\approx\dfrac{d^{\alpha}f(x)}{dx^{\alpha}},\nonumber\\
&A_{\alpha,2}^r u(x,t):=\dfrac{1}{\Gamma(3-\alpha)}\sum\limits_{k=0}^{\infty}\left(b_{k}+\dfrac{1}{2}c_{k}+\dfrac{1}{4}d_{k}\right)f(x+kh)\approx\dfrac{d^{\alpha}f(x)}{d(-x)^{\alpha}},
\end{align}
are of order $3-\alpha$ if $f\in V^{\gamma}$ for $\gamma>4+\alpha$.

The corresponding real part is
\[\textrm{Re}(g_{2}(\theta,\alpha))=8\sin^{3}\dfrac{\theta}{2}\Gamma(3-\alpha)\Big(g_{21}(\theta)H_{1}(\theta,\alpha)+g_{22}(\theta)H_{2}(\theta,\alpha)\Big),\]
where
\begin{align*}
&g_{21}(\theta)=\cos\dfrac{\theta}{2}+\sin\dfrac{\theta}{2}\sin\theta-\sin^{2}\dfrac{\theta}{2}\cos\dfrac{3\theta}{2},\\
&g_{22}(\theta)=\sin\dfrac{\theta}{2}-\sin\dfrac{\theta}{2}\cos\theta-\sin^{2}\dfrac{\theta}{2}\sin\dfrac{3\theta}{2}.
\end{align*}

As $g_{22}(\theta)=\sin^{3}(\theta/2)(-1+4\sin^{2}(\theta/2))$, we have $g_{22}(\theta)<0$ when $\theta$ is near $0$, such that $\textrm{Re}(g_{2}(\theta,\alpha))<0$ for small $\theta$ when $\alpha$ is near $1$. Nevertheless, the contours of $\textrm{Re}(g_{2}(\theta,\alpha))$ in Fig.~\ref{3a01} and Fig.~\ref{3a02} illustrate that it is non-negative for a wide range of $\alpha\in(0,0.986]$, where the $H_{1}$ and $H_{2}$ are evaluated with $n$ from $0$ to $1000$.

\begin{figure}[!hbtp]
\centering
\includegraphics[scale=0.32]{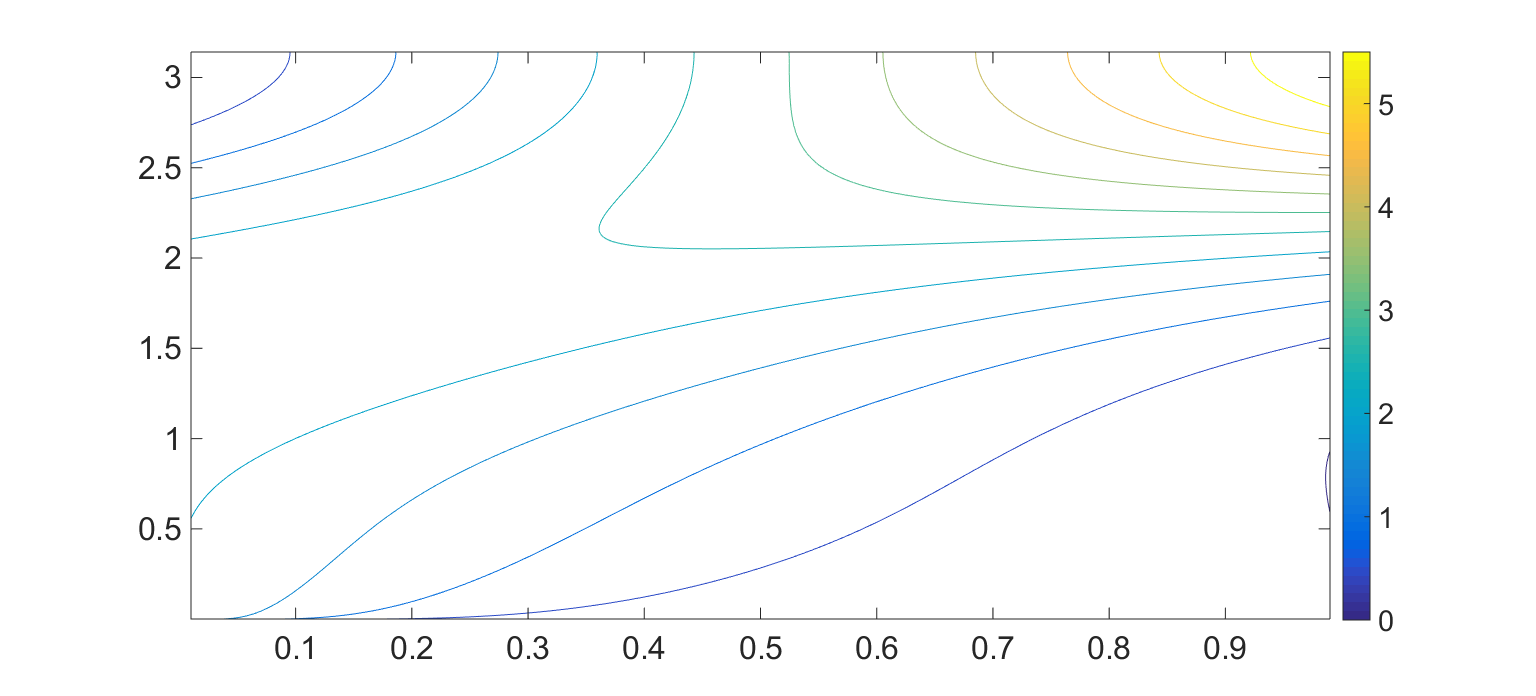}
\caption{\footnotesize{Contour of $\textrm{Re}(g_{2}(\theta,\alpha))$ with $\alpha\in[0.01,0.99],~\theta\in[0.001,3.141]$.}}\label{3a01}
\end{figure}
\begin{figure}[!hbtp]
\centering
\includegraphics[scale=0.32]{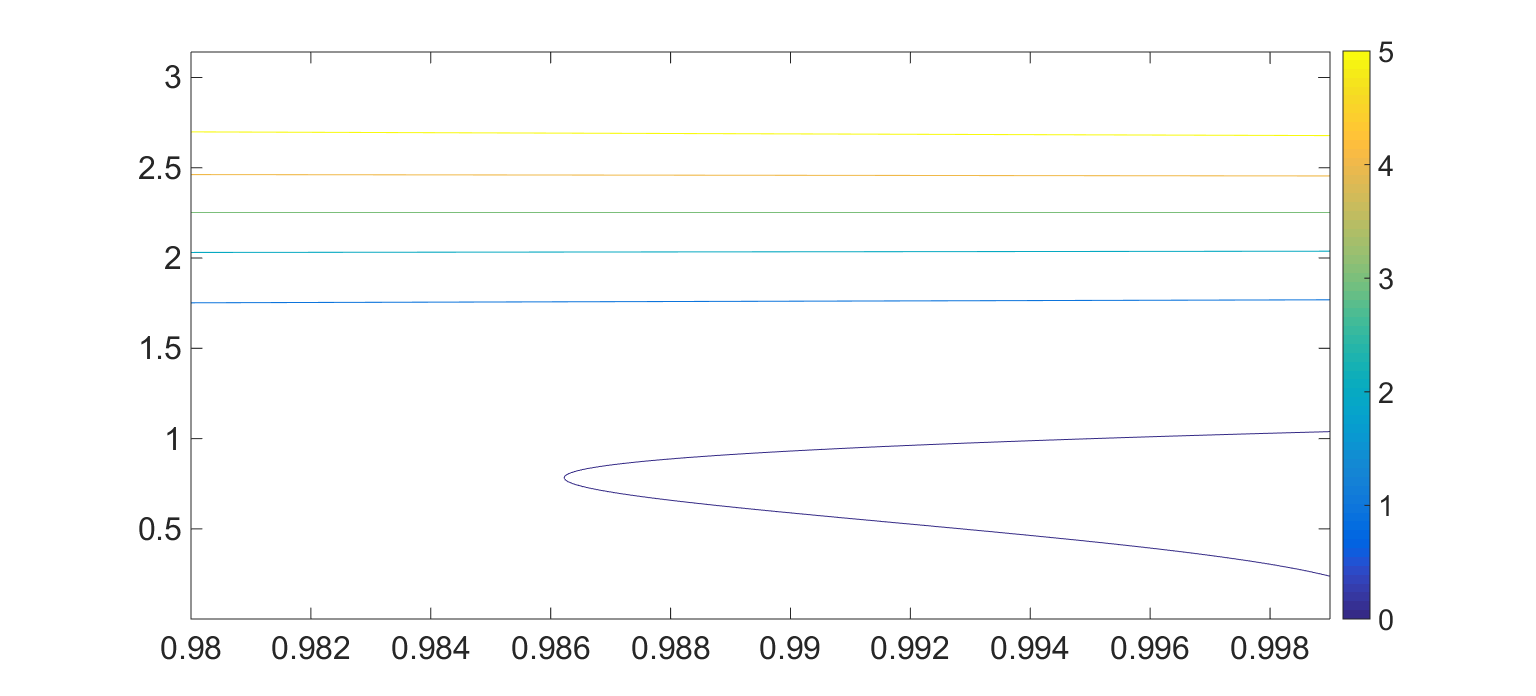}
\caption{\footnotesize{Contour of $\textrm{Re}(g_{2}(\theta,\alpha))$ with $\alpha\in[0.98,0.999],~\theta\in[0.001,3.141]$.}}\label{3a02}
\end{figure}

\subsection{Upwind-and-shift approximation}\label{sap2}
As an attempt to fill the gap where non-negative values emerge, we consider the following undetermined form,
\[g_{3}(\theta,\alpha)=AI_{1}+BI_{2}+CI_{3}+De^{i\theta}I_{1}+Ee^{i\theta}I_{2}+Fe^{i\theta}I_{3}+Pe^{i2\theta}I_{1}+Qe^{i2\theta}I_{2}+Re^{i2\theta}I_{3}.\]

Similarly, we have $g_{3}(\theta,\alpha)=(1+O(\theta^{3}))(\Gamma(3-\alpha)(-i\theta)^{\alpha}+O(\theta^{3}))$ for $\alpha\in(0,1)$ when
\[A=1-D-P,~B=\dfrac{1}{2}+D-E+2P-Q,~C=\dfrac{1}{4}+E-F-P+2Q-R.\]

Noticing that $I_{2}=(1-e^{i\theta})I_{1}$ and $I_{3}=(1-e^{i\theta})I_{2}=(1-e^{i\theta})^{2}I_{1}$, we then have
\[g_{3}(\theta,\alpha)=I_{1}+\dfrac{1}{2}I_{2}+\dfrac{1}{4}I_{3}-(F-Q)(1-e^{i\theta})I_{3}-R(1-e^{i2\theta})I_{3}\]
We note here that the terms involving $D,E,P$ have vanished.

Denote $c=\cos(\theta/2)$ and $s=\sin(\theta/2)$, then we have 
\[\textrm{Re}(g_{3}(\theta,\alpha))=8\sin^{3}\dfrac{\theta}{2}\Gamma(3-\alpha)\Big(g_{31}(\theta)H_{1}(\theta,\alpha)+g_{32}(\theta)H_{2}(\theta,\alpha)\Big),\]
with
\begin{align*}
g_{31}(\theta)=(6c-9c^{3}+4c^{5})+4(F-Q)&(-8c+32c^{3}-40c^{5}+16c^{7})\\
&~+4R(4c-56c^{3}+164c^{5}-176c^{7}+64c^{9}),\\
g_{32}(\theta)=s^{3}(-1+4s^{2})+8s^{3}(F-Q)&(-1+8s^{2}-8s^{4})\\
&~+16s^{3}R(-1+13s^{2}-28s^{4}+16s^{6}).
\end{align*}

Regrettably, negative values still emerge for any set of weights. In fact, if $\textrm{Re}(g_{3}(\theta,\alpha))\geq 0$ for all $(\theta,\alpha)\in[0,\pi]\times(0,1)$, then $g_{32}(\theta)\geq 0$ for all $\theta\in[0,\pi]$, hence
\begin{align*}
&\begin{cases}
\lim\limits_{\theta\to 0^{+}}g_{32}(\theta)=(-1-8(F-Q)-16R)\lim\limits_{s\to 0^{+}}s^{3}\geq 0,\\
g_{32}\left(\dfrac{\pi}{3}\right)=\dfrac{1}{2}(F-Q)+\dfrac{3}{2}R\geq 0,\\
g_{32}\left(\dfrac{\pi}{2}\right)=\dfrac{1}{2\sqrt{2}}(1+8(F-Q)+8R)\geq 0,\\
\end{cases}\\
\Rightarrow
&\begin{cases}
-1-8(F-Q)-16R\geq 0,&(i)\\
8(F-Q)+24R\geq 0,&(ii)\\
1+8(F-Q)+8R\geq 0,&(iii)\\
\end{cases}
\end{align*}
which has no solution as $(i)+(ii)$ gives $8R\geq 1$ and $(i)+(iii)$ gives $8R\leq 0$.

To select a set of weights such that $\textrm{Re}(g_{3}(\theta,\alpha))$ is non-negative for a wider range of $\alpha$, observing that $g_{32}$ is relatively more negative than $g_{31}$, we did a numerical search for $\arg\max_{F-Q,R}\min_{\theta}g_{32}(\theta)$. By taking $D=E=P=Q=0,~F=-0.23,~R=0.111$, such that $A=1,~B=0.5,~C=0.369$, we have $\textrm{Re}(g_{3}(\theta,\alpha))\geq 0$ for $(\theta,\alpha)\in[0,\pi]\times(0,0.997]$, as illustrated in Fig.~\ref{3a1} and Fig.~\ref{3a2}.

\begin{figure}[!tbhp]
\centering
\includegraphics[scale=0.32]{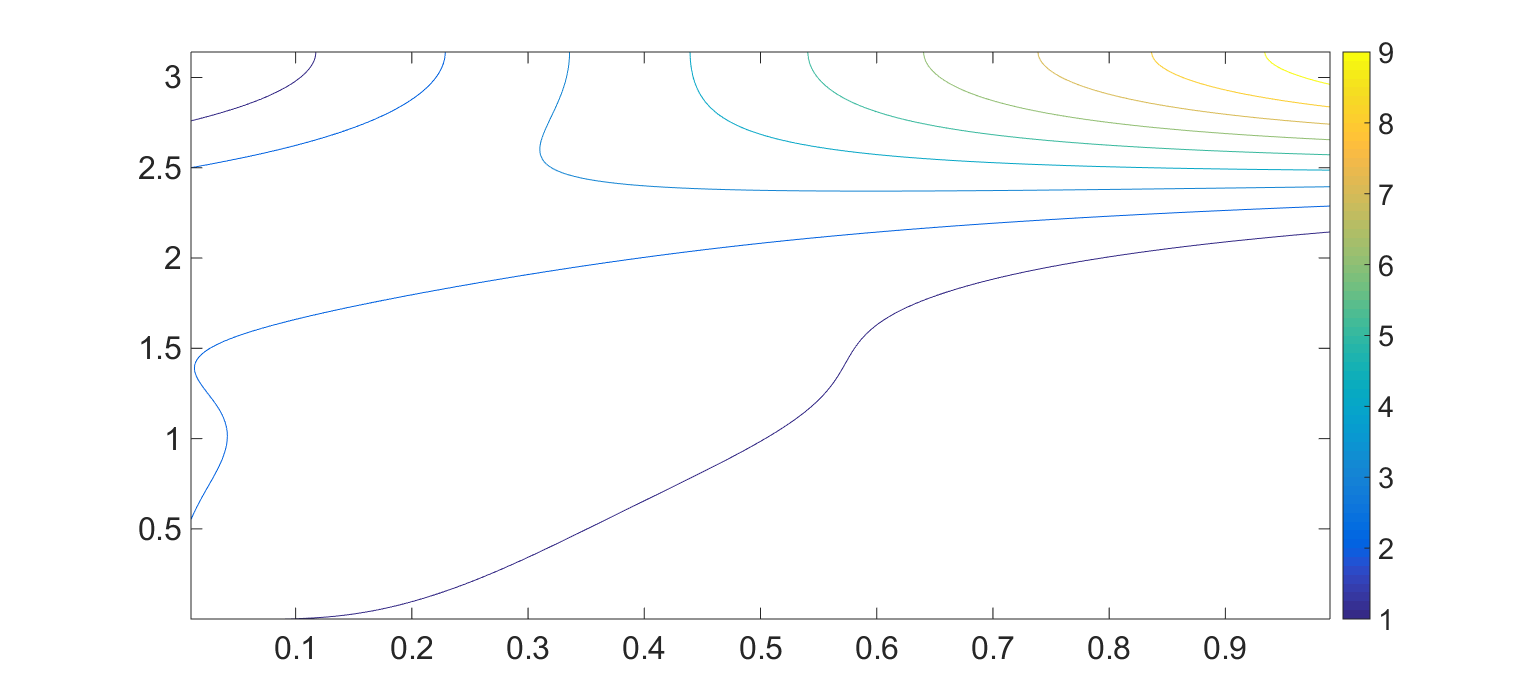}
\caption{\footnotesize{Contour of $\textrm{Re}(g_{3}(\theta,\alpha))$ with $\alpha\in[0.01,0.99],~\theta\in[0.001,3.141]$.}}\label{3a1}
\end{figure}
\begin{figure}[!tbhp]
\centering
\includegraphics[scale=0.32]{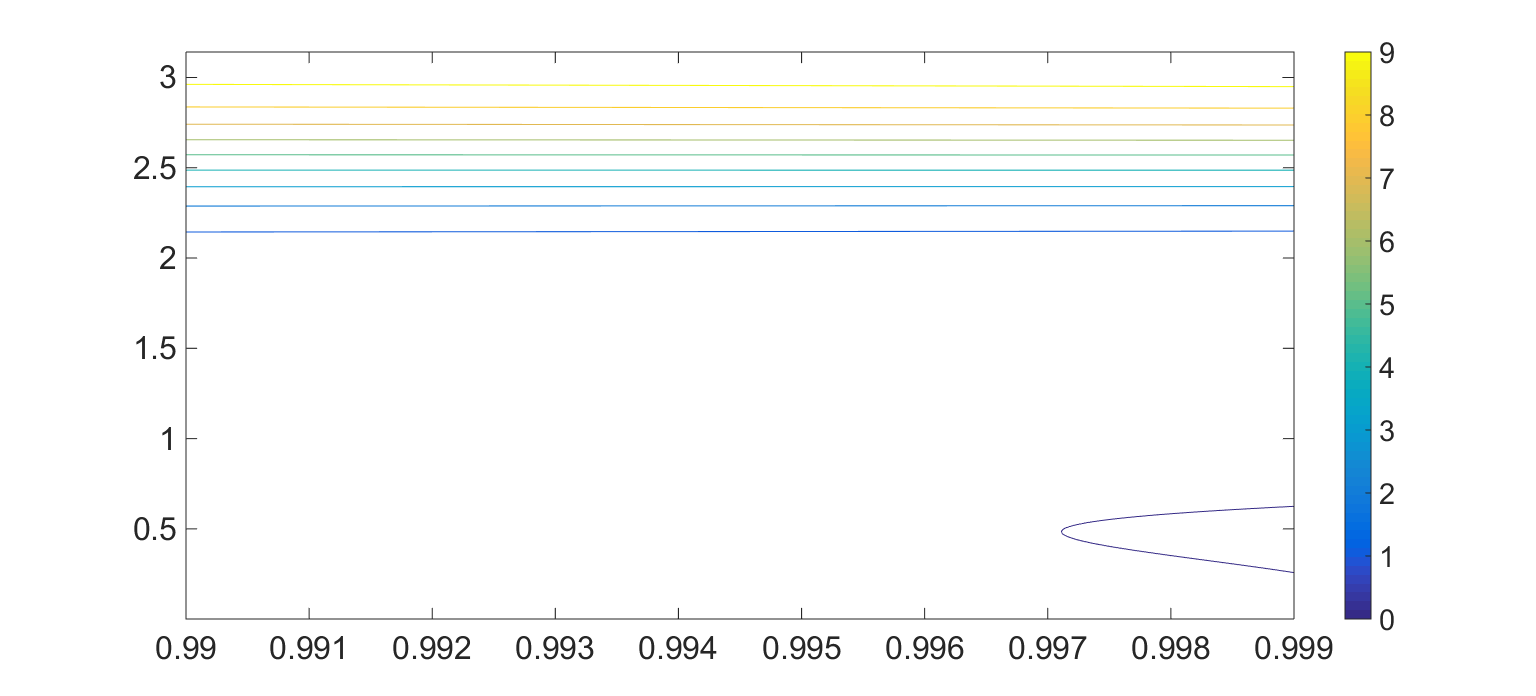}
\caption{\footnotesize{Contour of $\textrm{Re}(g_{3}(\theta,\alpha))$ with $\alpha\in[0.99,0.999],~\theta\in[0.001,3.141]$.}}\label{3a2}
\end{figure}

The corresponding approximations
\begin{align}\label{ap2}
&A_{\alpha,3}^l u(x,t):=\dfrac{1}{\Gamma(3-\alpha)}\left(\sum\limits_{k=0}^{\infty}\left(b_{k}+0.5c_{k}+0.369d_{k}\right)f(x-kh)\right.\nonumber\\
&\left.+\sum\limits_{k=0}^{\infty}-0.23d_{k}f(x-(k+1)h)+\sum\limits_{k=0}^{\infty}0.111d_{k}f(x-(k+2)h)\right)\approx\dfrac{d^{\alpha}f(x)}{dx^{\alpha}},\nonumber\\
&A_{\alpha,3}^r u(x,t):=\dfrac{1}{\Gamma(3-\alpha)}\left(\sum\limits_{k=0}^{\infty}\left(b_{k}+0.5c_{k}+0.369d_{k}\right)f(x+kh)\right.\nonumber\\
&\left.+\sum\limits_{k=0}^{\infty}-0.23d_{k}f(x+(k+1)h)+\sum\limits_{k=0}^{\infty}0.111d_{k}f(x+(k+2)h)\right)\approx\dfrac{d^{\alpha}f(x)}{d(-x)^{\alpha}},
\end{align}
are of order $3-\alpha$ if $f\in V^{\gamma}$ for $\gamma>4+\alpha$.

\begin{remark}
Aside from the considered forms, the possibility that other weighted forms may have non-negative real parts for all $\alpha\in(0,1)$ is not excluded. On the other hand, higher order approximations can be studied by considering higher powers and higher order divided differences in the definition of $a_{k}$, we remark here that the third order and fourth order multi-term upwind approximations without shifts are observed inadequate for being non-negative.
\end{remark}

\section{Finite difference scheme and the discrete linear system}\label{sec3}
In this section, the space fractional convection equation (\ref{fde}) is discretized by implicit finite difference method and structure of the discrete linear system is observed.

Consider uniform grid points $t_{n+1}=t_n+\tau$, $n=0,1,\ldots$, with $\tau>0$, and $x_{i+1}=x_i+h$, $i=0,1,\ldots$, with $h>0$.
Let
$
u^{n}_i=u(x_i,t_n)$ and $q^{n+1/2}_i=q(x_i,t_{n+1/2}), 
$
 for $i,n=0,1,2,\ldots$. Applying the
Crank-Nicolson technique to the fractional convection equation (\ref{fde}) with approximations
(\ref{2ndapp}), (\ref{ap1}), (\ref{ap2}), and denote operators $\delta_{m}^{\alpha,\beta}=(1+\beta) A_{\alpha,m}^l/2+(1-\beta) A_{\alpha,m}^r/2$ for $m=1,2,3$, we have
$$
\frac{u_{i}^{n+1}-u_{i}^{n}}{\tau}=- \frac{D}{2} \delta_{m}^{\alpha,\beta} u_{i}^{n}- \frac{D}{2} \delta_{m}^{\alpha,\beta} u_{i}^{n+1}+q_{i}^{n+1 / 2}+R_{i}^{n+1/2} 
$$

Ignoring the local error term $R_{i}^{n+1/2}$, then for each $m$, we get the following numerical scheme

\begin{equation}\label{scheme}
\left(1+\frac{D}{2 } \tau \delta_{m}^{\alpha,\beta}\right) U_{i}^{n+1}=\left(1-\frac{D}{2} \tau \delta_{m}^{\alpha,\beta}\right) U_{i}^{n}+\tau q_{i}^{n+1 / 2},
\end{equation}
where $U_i^n$ is a numerical approximation to $u_i^n$.

 Assume the nodal points $U_{i}^{n}$ are zero except for $i=1, \ldots, N$. 
 Introducing the vectors $\mathbf{U}^{n}=$ $\left[U_{1}^{n}, \ldots, U_{N}^{n}\right]^{\top}$ and $\mathbf{q}^{n+1/2}=$ $\left[q_{1}^{n+1/2}, \ldots, q_{N}^{n+1/2}\right]^{\top}$, the identity matrix $\mathbf{I}$, and the lower triangular Toeplitz matrices \cite{RchanJin} $\mathbf{B}_{1},~\mathbf{B}_{2},\mathbf{B}_{3},~\mathbf{J}$, with first columns $[b_{0},b_{1},\ldots,b_{N-1}]^{\top}$, $[c_{0},c_{1},\ldots,c_{N-1}]^{\top}$, $[d_{0},d_{1},\ldots,d_{N-1}]^{\top}$, $[0,1,0,0,\ldots,0]^{\top}$ respectively, the scheme~(\ref{scheme}) can be written in matrix form
\begin{equation}\label{system}
\left(\mathbf{I}+ \mu_{\alpha} \mathbf{B}_{m}^{\alpha,\beta}\right) \mathbf{U}^{n+1}=\left(\mathbf{I}- \mu_{\alpha} \mathbf{B}_{m}^{\alpha,\beta}\right) \mathbf{U}^{n}+\tau \mathbf{q}^{n+1/2},
\end{equation}
where $\mu_{\alpha}=D \tau/(2\Gamma(3-\alpha)h^{\alpha})$ and $\mathbf{B}_{m}^{\alpha,\beta}=(1+\beta) \widehat{\mathbf{B}}_{m}/2+(1-\beta) \widehat{\mathbf{B}}_{m}^{\top}/2$, with 
\begin{align*}
&\widehat{\mathbf{B}}_{1}=2(\mathbf{B}_{1}+\mathbf{B}_{2}/2),\\
&\widehat{\mathbf{B}}_{2}=2(\mathbf{B}_{1}+\mathbf{B}_{2}/2+\mathbf{B}_{3}/4),\\
&\widehat{\mathbf{B}}_{3}=2(\mathbf{B}_{1}+0.5\mathbf{B}_{2}+0.369\mathbf{B}_{3}-0.23\mathbf{J}\mathbf{B}_{3}+0.111\mathbf{J}^{2}\mathbf{B}_{3}).
\end{align*}

\section{Stability and convergence analysis}\label{sec4}
In this section, the numerical scheme (\ref{scheme}) with $m=1$ is shown to be unconditionally stable in the sense that the error in each time level is bounded by the error in the initial time level, and the convergence rate is $O(\tau^2+h^2)$. The cases for $m=2,3$ are also discussed. The following lemma is the key of the analysis.
\begin{lemma} \label{spd}
The symmetric matrix $\widehat{\mathbf{B}}_{1}+\widehat{\mathbf{B}}_{1}^{\top}$ is positive semidefinite.
\end{lemma}

\begin{proof}
From \cite{Jinbook}, the symmetric Toeplitz matrix $\widehat{\mathbf{B}}_{1}+\widehat{\mathbf{B}}_{1}^{\top}$ is positive semidefinite if its generating function is non-negative. Since the generating function is given as 
\begin{equation*}
\mathring{g}_{1}(\theta,\alpha):=g_{1}(\theta,\alpha)+g_{1}(-\theta,\alpha)=2\sum_{k=0}^{\infty}(2b_{k}+c_{k})\cos(k\theta),
\end{equation*}
by Lemma \ref{L2}, we have $\mathring{g}_{1}(\theta,\alpha)\geq 0$, hence $\widehat{\mathbf{B}}_{1}+\widehat{\mathbf{B}}_{1}^{\top}$ is positive semidefinite.
\end{proof}

Now, let $\mathbf{A}_{1}=\left(\mathbf{I}+\mu_{\alpha}\mathbf{B}_{1}^{\alpha,\beta}\right)^{-1}\left(\mathbf{I}-\mu_{\alpha}\mathbf{B}_{1}^{\alpha,\beta}\right)$. Then the matrix form (\ref{system}) can be rewritten as the form $\mathbf{U}^{n} =\mathbf{A}_{1} \mathbf{U}^{n-1}+\tau(\mathbf{I}+\mu_{\alpha}\mathbf{B}_{1}^{\alpha,\beta})^{-1}\mathbf{q}^{n-1/2}$, and we have the following lemma.

\begin{lemma}\label{L4}
The matrices $\mathbf{A}_{1}$ and $\left(\mathbf{I}+\mu_{\alpha}\mathbf{B}_{1}^{\alpha,\beta}\right)^{-1}$ satisfy the norm inequalities $\|\mathbf{A}_{1}\|_2\leq 1$ and $\left\|\left(\mathbf{I}+\mu_{\alpha}\mathbf{B}_{1}^{\alpha,\beta}\right)^{-1}\right\|_{2}\leq 1$ uniformly.
\end{lemma}

\begin{proof}
Let $(\mathbf{v},\lambda)$ be an eigenpair of the matrix 
\begin{equation*}
\mathbf{F}=\left(\mathbf{I}+(\mu_{\alpha}\mathbf{B}_{1}^{\alpha,\beta})^{\top}\right)^{-1}(\mathbf{I}+\mu_{\alpha}\mathbf{B}_{1}^{\alpha,\beta})^{-1}(\mathbf{I}-\mu_{\alpha}\mathbf{B}_{1}^{\alpha,\beta})\left(\mathbf{I}-(\mu_{\alpha}\mathbf{B}_{1}^{\alpha,\beta})^{\top}\right),
\end{equation*}
then we have $0\leq\lambda\leq 1$ from Lemma \ref{spd} and the identity
\begin{align*}
\mathbf{v}^{\top}(\mathbf{I}+\mu_{\alpha}\mathbf{B}_{1}^{\alpha,\beta})\left(\mathbf{I}+(\mu_{\alpha}\mathbf{B}_{1}^{\alpha,\beta})^{\top}\right) \mathbf{v}-\mathbf{v}^{\top}(\mathbf{I}-\mu_{\alpha}\mathbf{B}_{1}^{\alpha,\beta})&\left(\mathbf{I}-(\mu_{\alpha}\mathbf{B}_{1}^{\alpha,\beta})^{\top}\right)\mathbf{v}\\
&=2\mu_{\alpha}\mathbf{v}^{\top}(\mathbf{B}_{1}+\mathbf{B}_{1}^{\top})\mathbf{v}.
\end{align*}

As $\mathbf{S T}$ and $\mathbf{T S}$ have the same set of eigenvalues for all square matrices $\mathbf{S}, \mathbf{T}$, we have

\begin{align*}
&\lambda_{\max }(\mathbf{F})\\
&=\lambda_{\max }\left(\left(\mathbf{I}-(\mu_{\alpha}\mathbf{B}_{1}^{\alpha,\beta})^{\top}\right)\left(\mathbf{I}+(\mu_{\alpha}\mathbf{B}_{1}^{\alpha,\beta})^{\top}\right)^{-1}(\mathbf{I}+\mu_{\alpha}\mathbf{B}_{1}^{\alpha,\beta})^{-1}(\mathbf{I}-\mu_{\alpha}\mathbf{B}_{1}^{\alpha,\beta})\right)\\
&=\lambda_{\max }\left(\mathbf{A}_{1}^{\top} \mathbf{A}_{1}\right).
\end{align*}

Summing up, we have $\|\mathbf{A}_{1}\|_{2}=\sqrt{\lambda_{\max }\left(\mathbf{A}_{1}^{\top} \mathbf{A}_{1}\right)} \leq 1$. The case for the matrix $\left(\mathbf{I}+\mu_{\alpha}\mathbf{B}_{1}^{\alpha,\beta}\right)^{-1}$ can be obtained similarly.
\end{proof}

Now, we can obtain the following convergence result of the numerical scheme (\ref{scheme}) with $m=1$.

\begin{theorem}\label{T1}
Suppose the solution $u(x,t)$ of the fractional convection equation (\ref{fde}) satisfies the following properties.
\begin{itemize}
\item
$\partial^{2}u(x,t)/\partial t^{2}$ is bounded on $\mathbb{R}\times(0,\infty)$,
\item
for some $\gamma>3+\alpha$ and $\gamma'>1$, 
\begin{equation*}
\dfrac{\partial^{s'}}{\partial t^{s'}}\dfrac{\partial^{s}u(x,t)}{\partial(\pm x)^{s}}\in C(\mathbb{R}\times(0,\infty))\cap L^{1}(\mathbb{R}\times(0,\infty)),~\forall s\in[0,\gamma],~s'\in[0,\gamma']
\end{equation*}
and vanish at infinity $\forall s\in[0,\gamma-1],~s'\in[0,\gamma'-1]$.
\end{itemize}
Then the finite difference scheme (\ref{scheme}) with $m=1$ is convergent, and the order of accuracy is $O(\tau^2+h^2)$ in discrete 2-norm.
\end{theorem}

\begin{proof}
By Taylor's expansion for the temporal derivative, and by similar arguments described in the proof of Proposition~\ref{cs} for the spatial derivatives, it can be shown that the ignored error term in scheme~(\ref{scheme}) is $R_{i}^{n+1/2}=O(\tau^{2}+h^{2})$, so that for $\mathbf{R}^{n+1/2}=[R_{1}^{n+1/2},\ldots,R_{N}^{n+1/2}]^{\top}$, we have $h^{1/2}\left\|\mathbf{R}^{n+1/2}\right\|_{2}\leq C_{6}(\tau^{2}+h^{2})$ for some $C_{6}>0$ independent of $n,\tau,h$.

Let $\mathbf{U}^{n}$ be the solution vector of the equation~(\ref{fde}) and $\tilde{\mathbf{U}}^{n}$ be the solution vector of the scheme~(\ref{scheme}), we have 
\begin{align*}
&\mathbf{U}^{n+1} =\mathbf{A}_{1} \mathbf{U}^{n}+\tau(\mathbf{I}+\mu_{\alpha}\mathbf{B}_{1}^{\alpha,\beta})^{-1}\mathbf{q}^{n+1/2}+\tau(\mathbf{I}+\mu_{\alpha}\mathbf{B}_{1}^{\alpha,\beta})^{-1}\mathbf{R}^{n+1/2},~\textrm{and}\\
&\tilde{\mathbf{U}}^{n+1} =\mathbf{A}_{1}\tilde{\mathbf{U}}^{n}+\tau(\mathbf{I}+\mu_{\alpha}\mathbf{B}_{1}^{\alpha,\beta})^{-1}\mathbf{q}^{n+1/2}.
\end{align*}

Also, denote $E_n=h^{1 / 2}\left\|\mathbf{U}^{n}-\tilde{\mathbf{U}}^{n}\right\|_{2}$, by Lemma \ref{L4}, we get that
\begin{align*}
E_{n+1}&\leq h^{1/2}\left\|\mathbf{A}_{1}\left(\mathbf{U}^{n}-\tilde{\mathbf{U}}^{n}\right)\right\|_{2}+\tau h^{1/2}\left\|(\mathbf{I}+\mu_{\alpha}\mathbf{B}_{1}^{\alpha,\beta})^{-1}\mathbf{R}^{n+1/2}\right\|_{2}\\
&\leq h^{1 / 2}\left\|\mathbf{U}^{n}-\tilde{\mathbf{U}}^{n}\right\|_{2}+\tau h^{1/2}\left\|\mathbf{R}^{n+1/2}\right\|_{2}\\
&\leq E_{n}+C_{6}\tau(\tau^{2}+h^{2})\\
&\leq\cdots\\
&\leq E_{0}+C_{6}(n+1)\tau(\tau^{2}+h^{2})\\
&\leq C_{6}T(\tau^{2}+h^{2}),~\forall n,
\end{align*}
where $T$ is the length of the time interval.
\end{proof}

For the numerical scheme~(\ref{scheme}) with $m=2,3$, notice that 
\begin{equation*}
\mathring{g}_{m}(\theta,\alpha):=2(g_{m}(\theta,\alpha)+g_{m}(-\theta,\alpha))=4\textrm{Re}(g_{m}(\theta,\alpha))
\end{equation*}
is the generating function of the symmetric Toeplitz matrix $\widehat{\mathbf{B}}_{m}+\widehat{\mathbf{B}}_{m}^{\top}$, then relying on the numerical observations in Section \ref{sap1} and Section \ref{sap2}, Lemma~\ref{spd}, Lemma~\ref{L4}, and Theorem~\ref{T1} can be similarly obtained as follows.

\begin{theorem}
For $m=2$ with $\alpha\in(0,0.986]$, or $m=3$ with $\alpha\in(0,0.997]$, we have the following results.
\begin{itemize}
\item
The matrix $\widehat{\mathbf{B}}_{m}+\widehat{\mathbf{B}}_{m}^{\top}$ is symmetric positive semidefinite.
\item
For $\mathbf{A}_{m}=\left(\mathbf{I}+\mu_{\alpha}\mathbf{B}_{m}^{\alpha,\beta}\right)^{-1}\left(\mathbf{I}-\mu_{\alpha}\mathbf{B}_{m}^{\alpha,\beta}\right)$, we have
\[\|\mathbf{A}_{m}\|_2\leq 1\textrm{ and }\left\|\left(\mathbf{I}+\mu_{\alpha}\mathbf{B}_{m}^{\alpha,\beta}\right)^{-1}\right\|_{2}\leq 1.\]
\item
Suppose the solution $u(x,t)$ of the fractional convection equation (\ref{fde}) satisfies the following properties.
\begin{itemize}
\item
$\partial^{2}u(x,t)/\partial t^{2}$ is bounded on $\mathbb{R}\times(0,\infty)$,
\item
for some $\gamma>4+\alpha$ and $\gamma'>1$, 
\begin{equation*}
\dfrac{\partial^{s'}}{\partial t^{s'}}\dfrac{\partial^{s}u(x,t)}{\partial(\pm x)^{s}}\in C(\mathbb{R}\times(0,\infty))\cap L^{1}(\mathbb{R}\times(0,\infty)),~\forall s\in[0,\gamma],~s'\in[0,\gamma']
\end{equation*}
and vanish at infinity $\forall s\in[0,\gamma-1],~s'\in[0,\gamma'-1]$.
\end{itemize}
Then the finite difference scheme (\ref{scheme}) is convergent, and the order of accuracy is $O(\tau^2+h^{3-\alpha})$ in discrete 2-norm.
\end{itemize}
\end{theorem}

\section{Numerical results}\label{sec5}
In this section, numerical experiments are performed to show the effectiveness of the high order finite difference schemes for the fractional convection equations with order $0<\alpha<1$. 

In the computations, we take $M$ temporal grid points and $N$ spatial grid points, with $N=M$ for scheme~(\ref{scheme}) where approximation~(\ref{2ndapp}) is adopted, and $N=\lfloor M^{2/(3-\alpha)}\rfloor<M$ for scheme~(\ref{scheme}) where approximation~(\ref{ap1}) or (\ref{ap2}) is adopted, so that second order convergence is expected.

The results obtained from the linear systems being solved by the direct method and the GMRES method will be provided. For the direct method, the inverse of the time-independent matrix is computed before time-marching, so that only matrix-vector multiplication is required for the remaining part. For the GMRES method, we set restart $=10$, maxit $=N-1$, with $tol$ to be specified. Zero vector is chosen to be the initial guess and the stopping criterion is the 2-norm of the residual less than that of the right-hand side vector times a given $tol$. 

In the following, two numerical examples are tested. The first example has exact solution so that error can be computed directly and order of accuracy can be confirmed. The second example is to approximate a distribution function related to L\'{e}vy flights. The accuracy is confirmed by a known probability density function for particular value of $\alpha$ and $\beta$.

In the tables, each error is computed between the numerical and exact solutions at the final time level, and the discrete 2-norm of the errors $E_{h^{2}}$ are reported for different numbers of temporal grid points $M$. ``Rate" denotes the convergence rates of the numerical solutions, ``Time" denotes the CPU time in seconds, and ``Iter" denotes the average numbers of iterations for the convergence of the GMRES method over all time levels. 

In addition, all experiments are carried out by MATLAB R2014b with configuration Intel(R) Core(TM) i5-8300H CPU 2.30 GHz and 8GB RAM.

\begin{example}\label{ex1}
Consider the fractional convection equation (\ref{fde}) defined on $(x,t)\in[0,2]\times[0,1]$, with $D=1$, exact solution $u(x, t)=\exp(-t) x^{4}(2-x)^{4}$, initial condition $u_{0}(x)=x^{4}(2-x)^{4}$, corresponding source term $q(x, t)$, and boundary condition $u(x, t)=0$ for $x$ outside $(0,2)$. We also take $tol=\tau^{3}$ to maintain second order convergence. The results are presented in Table~\ref{tab1} to Table~\ref{tab4}.
\end{example}

Table~\ref{tab1} to Table~\ref{tab4} illustrate that for a wide range of $\alpha$, the schemes~(\ref{scheme}) with approximations~(\ref{ap1}) and~(\ref{ap2}) achieve second order convergence with fewer grid points, hence less time, than the scheme~(\ref{scheme}) with approximation~(\ref{2ndapp}). The second order convergence is achieved with $h=O(\tau^{2/(3-\alpha)})$, which agrees with the derivation that the spatial convergence order is of $3-\alpha$. Besides, there are noticeable drops of convergence order for approximation~(\ref{ap1}) when $\alpha=0.99$, this can be attributed to the negativity of the generating function $g_{2}$ for $\alpha\in(0,986,1)$ as discussed in section~\ref{sec2}, hence the instability of the scheme; while there is no observable drop of convergence order for approximation~(\ref{ap2}). 

On the other hand, the accuracy of the direct method and GMRES method are almost the same, except for the case with approximation~(\ref{ap1}) when $\alpha=0.99$, which probably follows from the numerical instability as mentioned above. The CPU times of the GMRES method, or the increasing rates, are smaller than that of the direct method, especially when $M$ is large. Moreover, the numbers of iterations of the GMRES method are very small for all different $\alpha$, $\beta$, and $M$. Nevertheless, when $\alpha$ is near $1$, the CPU times and the numbers of iterations of the GMRES method with approximations~(\ref{ap1}) and~(\ref{ap2}) are larger than that with approximation~(\ref{2ndapp}).

\begin{table}[H]
\captionsetup{font=footnotesize}
\caption{Results for Example~\ref{ex1} with approximation~(\ref{2ndapp}) and $\alpha=0.1,0.9$.}\label{tab1}
\centering
\fontsize{8}{8}\selectfont
\begin{tabular}{@{\extracolsep{5pt}}cccccccc}
\hline  & & \multicolumn{2}{c}{Direct method} & \multicolumn{4}{c}{GMRES method}\\     \cline{3-4} \cline{5-8}
$(\alpha,\beta)$ & $M$ & $E_{h^{2}}$ & Time & $E_{h^{2}}$ & Rate & Time & Iter\\
\hline
\multirow{5}{*}{$(0.1,0.2)$}														
&	$2^{9}$	&	9.40E-06	&	0.05 	&	9.37E-06	&	-	&	0.67 	&	2.00 	\\
&	$2^{10}$	&	2.35E-06	&	0.97 	&	2.34E-06	&	2.00 	&	0.67 	&	2.00 	\\
&	$2^{11}$	&	5.88E-07	&	8.60 	&	5.88E-07	&	1.99 	&	2.16 	&	3.00 	\\
&	$2^{12}$	&	1.47E-07	&	66.85 	&	1.47E-07	&	2.00 	&	11.45 	&	3.00 	\\
&	$2^{13}$	&	3.68E-08	&	545.28 	&	3.68E-08	&	2.00 	&	48.92 	&	3.00 	\\ \hline
\multirow{5}{*}{$(0.1,0.8)$}														
&	$2^{9}$	&	9.45E-06	&	0.06 	&	9.45E-06	&	-	&	0.32 	&	3.00 	\\
&	$2^{10}$	&	2.36E-06	&	1.05 	&	2.36E-06	&	2.00 	&	0.78 	&	3.00 	\\
&	$2^{11}$	&	5.92E-07	&	8.47 	&	5.92E-07	&	2.00 	&	2.19 	&	3.00 	\\
&	$2^{12}$	&	1.48E-07	&	67.22 	&	1.48E-07	&	2.00 	&	10.93 	&	3.00 	\\
&	$2^{13}$	&	3.70E-08	&	544.94 	&	3.70E-08	&	2.00 	&	47.64 	&	3.00 	\\ \hline
\multirow{5}{*}{$(0.9,0.2)$}														
&	$2^{9}$	&	1.15E-05	&	0.05 	&	1.15E-05	&	-	&	0.47 	&	5.00 	\\
&	$2^{10}$	&	2.80E-06	&	1.11 	&	2.81E-06	&	2.03 	&	1.10 	&	5.00 	\\
&	$2^{11}$	&	6.91E-07	&	8.36 	&	6.92E-07	&	2.02 	&	3.24 	&	5.00 	\\
&	$2^{12}$	&	1.71E-07	&	67.28 	&	1.71E-07	&	2.02 	&	18.87 	&	6.00 	\\
&	$2^{13}$	&	4.24E-08	&	541.02 	&	4.25E-08	&	2.01 	&	82.03 	&	6.00 	\\ \hline
\multirow{5}{*}{$(0.9,0.8)$}														
&	$2^{9}$	&	2.36E-05	&	0.05 	&	2.36E-05	&	-	&	0.45 	&	5.00 	\\
&	$2^{10}$	&	5.86E-06	&	0.95 	&	5.86E-06	&	2.01 	&	1.10 	&	5.00 	\\
&	$2^{11}$	&	1.45E-06	&	8.40 	&	1.45E-06	&	2.01 	&	3.52 	&	6.00 	\\
&	$2^{12}$	&	3.60E-07	&	67.00 	&	3.60E-07	&	2.01 	&	18.91 	&	6.00 	\\
&	$2^{13}$	&	8.94E-08	&	540.75 	&	8.94E-08	&	2.01 	&	82.00 	&	6.00 	\\ \hline
\end{tabular}
\end{table}

\begin{table}[H]
\captionsetup{font=footnotesize}
\caption{Results for Example~\ref{ex1} with approximation~(\ref{2ndapp}) and $\alpha=0.91,0.99,0.991,0.999$.}\label{tab2}
\centering
\fontsize{8}{8}\selectfont
\begin{tabular}{@{\extracolsep{5pt}}cccccccc}
\hline  & & \multicolumn{2}{c}{Direct method} & \multicolumn{4}{c}{GMRES method}\\     \cline{3-4} \cline{5-8}
$(\alpha,\beta)$ & $M$ & $E_{h^{2}}$ & Time & $E_{h^{2}}$ & Rate & Time & Iter\\
\hline
\multirow{5}{*}{$(0.95,0.2)$}														
&	$2^{9}$	&	1.32E-05	&	0.05 	&	1.32E-05	&	-	&	0.47 	&	5.00 	\\
&	$2^{10}$	&	3.24E-06	&	1.06 	&	3.24E-06	&	2.02 	&	1.11 	&	5.00 	\\
&	$2^{11}$	&	8.02E-07	&	8.51 	&	8.02E-07	&	2.02 	&	3.51 	&	6.00 	\\
&	$2^{12}$	&	1.99E-07	&	67.53 	&	1.99E-07	&	2.01 	&	19.23 	&	6.00 	\\
&	$2^{13}$	&	4.94E-08	&	541.88 	&	4.94E-08	&	2.01 	&	94.87 	&	7.00 	\\ \hline
\multirow{5}{*}{$(0.95,0.8)$}														
&	$2^{9}$	&	2.80E-05	&	0.05 	&	2.80E-05	&	-	&	0.48 	&	5.00 	\\
&	$2^{10}$	&	6.97E-06	&	1.11 	&	6.97E-06	&	2.01 	&	1.28 	&	6.00 	\\
&	$2^{11}$	&	1.73E-06	&	8.46 	&	1.73E-06	&	2.01 	&	3.52 	&	6.00 	\\
&	$2^{12}$	&	4.31E-07	&	67.14 	&	4.31E-07	&	2.01 	&	21.68 	&	7.00 	\\
&	$2^{13}$	&	1.07E-07	&	542.22 	&	1.07E-07	&	2.01 	&	95.89 	&	7.00 	\\ \hline
\multirow{5}{*}{$(0.99,0.2)$}														
&	$2^{9}$	&	1.62E-05	&	0.05 	&	1.62E-05	&	-	&	0.46 	&	5.00 	\\
&	$2^{10}$	&	4.05E-06	&	1.03 	&	4.05E-06	&	2.00 	&	1.10 	&	5.00 	\\
&	$2^{11}$	&	1.01E-06	&	8.45 	&	1.01E-06	&	2.00 	&	3.53 	&	6.00 	\\
&	$2^{12}$	&	2.53E-07	&	67.37 	&	2.53E-07	&	2.00 	&	19.10 	&	6.00 	\\
&	$2^{13}$	&	6.31E-08	&	541.65 	&	6.31E-08	&	2.00 	&	94.49 	&	7.00 	\\ \hline
\multirow{5}{*}{$(0.99,0.8)$}														
&	$2^{9}$	&	3.31E-05	&	0.06 	&	3.31E-05	&	-	&	0.45 	&	5.00 	\\
&	$2^{10}$	&	8.28E-06	&	0.99 	&	8.28E-06	&	2.00 	&	1.10 	&	5.00 	\\
&	$2^{11}$	&	2.07E-06	&	8.38 	&	2.07E-06	&	2.00 	&	3.50 	&	6.00 	\\
&	$2^{12}$	&	5.17E-07	&	67.69 	&	5.17E-07	&	2.00 	&	22.23 	&	7.00 	\\
&	$2^{13}$	&	1.29E-07	&	544.44 	&	1.29E-07	&	2.00 	&	95.90 	&	7.00 	\\ \hline
\multirow{5}{*}{$(0.995,0.2)$}														
&	$2^{9}$	&	1.68E-05	&	0.05 	&	1.68E-05	&	-	&	0.41 	&	4.00 	\\
&	$2^{10}$	&	4.20E-06	&	1.02 	&	4.20E-06	&	2.00 	&	1.08 	&	5.00 	\\
&	$2^{11}$	&	1.05E-06	&	8.35 	&	1.05E-06	&	2.00 	&	3.05 	&	5.00 	\\
&	$2^{12}$	&	2.62E-07	&	67.88 	&	2.62E-07	&	2.00 	&	18.85 	&	6.00 	\\
&	$2^{13}$	&	6.56E-08	&	539.40 	&	6.56E-08	&	2.00 	&	93.51 	&	7.00 	\\ \hline
\multirow{5}{*}{$(0.995,0.8)$}														
&	$2^{9}$	&	3.39E-05	&	0.06 	&	3.39E-05	&	-	&	0.39 	&	4.00 	\\
&	$2^{10}$	&	8.48E-06	&	1.18 	&	8.49E-06	&	2.00 	&	1.08 	&	5.00 	\\
&	$2^{11}$	&	2.12E-06	&	8.49 	&	2.12E-06	&	2.00 	&	3.52 	&	6.00 	\\
&	$2^{12}$	&	5.31E-07	&	67.64 	&	5.31E-07	&	2.00 	&	19.06 	&	6.00 	\\
&	$2^{13}$	&	1.33E-07	&	543.17 	&	1.33E-07	&	2.00 	&	93.53 	&	7.00 	\\ \hline
\multirow{5}{*}{$(0.999,0.2)$}														
&	$2^{9}$	&	1.73E-05	&	0.06 	&	1.73E-05	&	-	&	0.41 	&	4.00 	\\
&	$2^{10}$	&	4.33E-06	&	1.04 	&	4.33E-06	&	2.00 	&	0.97 	&	4.00 	\\
&	$2^{11}$	&	1.08E-06	&	8.44 	&	1.08E-06	&	2.00 	&	3.15 	&	5.00 	\\
&	$2^{12}$	&	2.71E-07	&	67.50 	&	2.71E-07	&	2.00 	&	16.43 	&	5.00 	\\
&	$2^{13}$	&	6.78E-08	&	542.77 	&	6.78E-08	&	2.00 	&	80.85 	&	6.00 	\\ \hline
\multirow{5}{*}{$(0.999,0.8)$}														
&	$2^{9}$	&	3.45E-05	&	0.05 	&	3.45E-05	&	-	&	0.40 	&	4.00 	\\
&	$2^{10}$	&	8.65E-06	&	1.05 	&	8.66E-06	&	2.00 	&	0.95 	&	4.00 	\\
&	$2^{11}$	&	2.17E-06	&	8.44 	&	2.17E-06	&	2.00 	&	3.05 	&	5.00 	\\
&	$2^{12}$	&	5.42E-07	&	67.38 	&	5.42E-07	&	2.00 	&	16.14 	&	5.00 	\\
&	$2^{13}$	&	1.36E-07	&	545.85 	&	1.36E-07	&	2.00 	&	80.60 	&	6.00 	\\ \hline
\end{tabular}
\end{table}

\begin{table}[H]
\captionsetup{font=footnotesize}
\caption{Results for Example~\ref{ex1} with approximation~(\ref{ap1}) and $\alpha=0.1,0.9,0.91,0.99$.}\label{tab3}
\centering
\fontsize{8}{8}\selectfont
\begin{tabular}{@{\extracolsep{5pt}}cccccccc}
\hline  & & \multicolumn{2}{c}{Direct method} & \multicolumn{4}{c}{GMRES method}\\     \cline{3-4} \cline{5-8}
$(\alpha,\beta)$ & $M$ & $E_{h^{2}}$ & Time & $E_{h^{2}}$ & Rate & Time & Iter\\
\hline
\multirow{5}{*}{$(0.1,0.2)$}														
&	$2^{9}$	&	6.50E-06	&	0.01 	&	6.75E-06	&	-	&	0.23 	&	2.00 	\\
&	$2^{10}$	&	1.43E-06	&	0.01 	&	1.58E-06	&	2.09 	&	0.45 	&	2.00 	\\
&	$2^{11}$	&	3.54E-07	&	0.04 	&	3.54E-07	&	2.16 	&	1.49 	&	3.00 	\\
&	$2^{12}$	&	8.94E-08	&	0.16 	&	8.94E-08	&	1.99 	&	2.87 	&	3.00 	\\
&	$2^{13}$	&	2.21E-08	&	0.55 	&	2.21E-08	&	2.02 	&	6.76 	&	3.00 	\\ \hline
\multirow{5}{*}{$(0.1,0.8)$}														
&	$2^{9}$	&	2.34E-05	&	0.00 	&	2.34E-05	&	-	&	0.28 	&	3.00 	\\
&	$2^{10}$	&	5.49E-06	&	0.02 	&	5.49E-06	&	2.09 	&	0.55 	&	3.00 	\\
&	$2^{11}$	&	1.32E-06	&	0.05 	&	1.32E-06	&	2.05 	&	1.10 	&	3.00 	\\
&	$2^{12}$	&	3.20E-07	&	0.16 	&	3.20E-07	&	2.05 	&	2.85 	&	3.00 	\\
&	$2^{13}$	&	7.65E-08	&	0.55 	&	7.65E-08	&	2.07 	&	6.86 	&	3.00 	\\ \hline
\multirow{5}{*}{$(0.9,0.2)$}														
&	$2^{9}$	&	2.76E-06	&	0.03 	&	2.88E-06	&	-	&	0.46 	&	5.00 	\\
&	$2^{10}$	&	6.57E-07	&	0.26 	&	7.09E-07	&	2.02 	&	1.09 	&	5.00 	\\
&	$2^{11}$	&	1.62E-07	&	4.10 	&	1.84E-07	&	1.95 	&	3.82 	&	5.00 	\\
&	$2^{12}$	&	4.02E-08	&	31.19 	&	4.10E-08	&	2.16 	&	25.31 	&	6.00 	\\
&	$2^{13}$	&	1.00E-08	&	221.79 	&	1.03E-08	&	1.99 	&	96.29 	&	6.00 	\\ \hline
\multirow{5}{*}{$(0.9,0.8)$}														
&	$2^{9}$	&	6.53E-06	&	0.04 	&	6.54E-06	&	-	&	0.46 	&	5.00 	\\
&	$2^{10}$	&	1.62E-06	&	0.26 	&	1.62E-06	&	2.01 	&	1.13 	&	5.00 	\\
&	$2^{11}$	&	4.05E-07	&	4.23 	&	4.05E-07	&	2.00 	&	4.36 	&	6.00 	\\
&	$2^{12}$	&	1.01E-07	&	30.55 	&	1.01E-07	&	2.00 	&	25.29 	&	6.00 	\\
&	$2^{13}$	&	2.53E-08	&	221.98 	&	2.53E-08	&	2.00 	&	99.24 	&	6.00 	\\ \hline
\multirow{5}{*}{$(0.95,0.2)$}														
&	$2^{9}$	&	3.40E-06	&	0.04 	&	3.70E-06	&	-	&	0.61 	&	5.00 	\\
&	$2^{10}$	&	8.30E-07	&	0.58 	&	8.50E-07	&	2.12 	&	1.24 	&	6.00 	\\
&	$2^{11}$	&	2.06E-07	&	5.96 	&	2.17E-07	&	1.97 	&	3.64 	&	6.00 	\\
&	$2^{12}$	&	5.14E-08	&	45.21 	&	5.21E-08	&	2.06 	&	36.66 	&	7.00 	\\
&	$2^{13}$	&	1.28E-08	&	339.73 	&	1.32E-08	&	1.98 	&	114.27 	&	7.00 	\\ \hline
\multirow{5}{*}{$(0.95,0.8)$}														
&	$2^{9}$	&	7.48E-06	&	0.05 	&	7.48E-06	&	-	&	0.68 	&	6.00 	\\
&	$2^{10}$	&	1.86E-06	&	0.59 	&	1.86E-06	&	2.00 	&	1.23 	&	6.00 	\\
&	$2^{11}$	&	4.65E-07	&	5.89 	&	4.65E-07	&	2.00 	&	4.14 	&	7.00 	\\
&	$2^{12}$	&	1.16E-07	&	45.33 	&	1.16E-07	&	2.00 	&	36.49 	&	7.00 	\\
&	$2^{13}$	&	2.91E-08	&	340.69 	&	2.91E-08	&	2.00 	&	128.99 	&	8.00 	\\ \hline
\multirow{5}{*}{$(0.99,0.2)$}														
&	$2^{9}$	&	4.19E-06	&	0.06 	&	4.26E-06	&	-	&	0.49 	&	5.00 	\\
&	$2^{10}$	&	1.03E-06	&	0.88 	&	1.04E-06	&	2.03 	&	1.64 	&	6.00 	\\
&	$2^{11}$	&	2.58E-07	&	8.07 	&	2.62E-07	&	1.99 	&	5.64 	&	6.00 	\\
&	$2^{12}$	&	6.44E-08	&	62.02 	&	6.61E-08	&	1.99 	&	34.45 	&	7.00 	\\
&	$2^{13}$	&	2.49E-07	&	487.01 	&	8.38E-06	&	-6.99	&	174.55 	&	8.00 	\\ \hline
\multirow{5}{*}{$(0.99,0.8)$}														
&	$2^{9}$	&	8.44E-06	&	0.05 	&	8.45E-06	&	-	&	0.49 	&	5.00 	\\
&	$2^{10}$	&	2.11E-06	&	0.99 	&	2.11E-06	&	2.00 	&	1.64 	&	6.00 	\\
&	$2^{11}$	&	5.27E-07	&	8.55 	&	5.28E-07	&	2.00 	&	6.34 	&	7.00 	\\
&	$2^{12}$	&	1.32E-07	&	62.08 	&	1.32E-07	&	2.00 	&	38.61 	&	8.00 	\\
&	$2^{13}$	&	1.03E-07	&	489.97 	&	5.33E-06	&	-5.33	&	206.55 	&	9.45 	\\ \hline
\end{tabular}
\end{table}

\begin{table}[H]
\captionsetup{font=footnotesize}
\caption{Results for Example~\ref{ex1} with approximation~(\ref{ap2}) and $\alpha=0.1,0.9,0.991,0.999$.}\label{tab4}
\centering
\fontsize{8}{8}\selectfont
\begin{tabular}{@{\extracolsep{5pt}}cccccccc}
\hline  & & \multicolumn{2}{c}{Direct method} & \multicolumn{4}{c}{GMRES method}\\     \cline{3-4} \cline{5-8}
$(\alpha,\beta)$ & $M$ & $E_{h^{2}}$ & Time & $E_{h^{2}}$ & Rate & Time & Iter\\
\hline
\multirow{5}{*}{$(0.1,0.2)$}														
&	$2^{9}$	&	7.73E-06	&	0.01 	&	7.89E-06	&	-	&	0.23 	&	2.00 	\\
&	$2^{10}$	&	1.42E-06	&	0.01 	&	1.55E-06	&	2.34 	&	0.45 	&	2.00 	\\
&	$2^{11}$	&	3.22E-07	&	0.05 	&	3.22E-07	&	2.27 	&	1.12 	&	3.00 	\\
&	$2^{12}$	&	8.00E-08	&	0.15 	&	8.00E-08	&	2.01 	&	2.80 	&	3.00 	\\
&	$2^{13}$	&	2.00E-08	&	0.58 	&	2.00E-08	&	2.00 	&	6.79 	&	3.00 	\\ \hline
\multirow{5}{*}{$(0.1,0.8)$}														
&	$2^{9}$	&	2.17E-05	&	0.01 	&	2.17E-05	&	-	&	0.28 	&	3.00 	\\
&	$2^{10}$	&	5.13E-06	&	0.01 	&	5.13E-06	&	2.08 	&	0.56 	&	3.00 	\\
&	$2^{11}$	&	1.24E-06	&	0.03 	&	1.24E-06	&	2.05 	&	1.11 	&	3.00 	\\
&	$2^{12}$	&	3.00E-07	&	0.14 	&	3.00E-07	&	2.05 	&	2.89 	&	3.00 	\\
&	$2^{13}$	&	7.16E-08	&	0.55 	&	7.16E-08	&	2.06 	&	6.81 	&	3.00 	\\ \hline
\multirow{5}{*}{$(0.9,0.2)$}														
&	$2^{9}$	&	2.74E-06	&	0.04 	&	2.82E-06	&	-	&	0.56 	&	6.00 	\\
&	$2^{10}$	&	6.56E-07	&	0.35 	&	6.88E-07	&	2.03 	&	1.25 	&	6.00 	\\
&	$2^{11}$	&	1.62E-07	&	4.10 	&	1.75E-07	&	1.98 	&	4.34 	&	6.00 	\\
&	$2^{12}$	&	4.02E-08	&	30.62 	&	4.09E-08	&	2.09 	&	28.94 	&	7.00 	\\
&	$2^{13}$	&	1.00E-08	&	222.54 	&	1.03E-08	&	1.99 	&	109.99 	&	7.00 	\\ \hline
\multirow{5}{*}{$(0.9,0.8)$}														
&	$2^{9}$	&	6.54E-06	&	0.03 	&	6.54E-06	&	-	&	0.53 	&	6.00 	\\
&	$2^{10}$	&	1.62E-06	&	0.27 	&	1.62E-06	&	2.01 	&	1.25 	&	6.00 	\\
&	$2^{11}$	&	4.05E-07	&	4.18 	&	4.05E-07	&	2.00 	&	4.89 	&	7.00 	\\
&	$2^{12}$	&	1.01E-07	&	30.57 	&	1.01E-07	&	2.00 	&	29.53 	&	7.00 	\\
&	$2^{13}$	&	2.52E-08	&	222.57 	&	2.53E-08	&	2.00 	&	123.35 	&	8.00 	\\ \hline
\multirow{5}{*}{$(0.995,0.2)$}														
&	$2^{9}$	&	4.30E-06	&	0.04 	&	4.32E-06	&	-	&	0.53 	&	6.00 	\\
&	$2^{10}$	&	1.07E-06	&	0.93 	&	1.07E-06	&	2.02 	&	1.86 	&	7.00 	\\
&	$2^{11}$	&	2.66E-07	&	8.58 	&	2.67E-07	&	2.00 	&	6.38 	&	7.00 	\\
&	$2^{12}$	&	6.64E-08	&	65.18 	&	6.66E-08	&	2.00 	&	38.34 	&	8.00 	\\
&	$2^{13}$	&	1.66E-08	&	510.03 	&	1.66E-08	&	2.00 	&	242.80 	&	9.00 	\\ \hline
\multirow{5}{*}{$(0.995,0.8)$}														
&	$2^{9}$	&	8.58E-06	&	0.05 	&	8.59E-06	&	-	&	0.55 	&	6.00 	\\
&	$2^{10}$	&	2.15E-06	&	0.84 	&	2.15E-06	&	2.00 	&	2.12 	&	8.00 	\\
&	$2^{11}$	&	5.36E-07	&	8.50 	&	5.36E-07	&	2.00 	&	7.96 	&	9.00 	\\
&	$2^{12}$	&	1.34E-07	&	65.07 	&	1.34E-07	&	2.00 	&	53.24 	&	11.00 	\\
&	$2^{13}$	&	3.35E-08	&	511.97 	&	3.35E-08	&	2.00 	&	324.88 	&	12.00 	\\ \hline
\multirow{5}{*}{$(0.999,0.2)$}														
&	$2^{9}$	&	4.41E-06	&	0.05 	&	4.43E-06	&	-	&	0.51 	&	5.00 	\\
&	$2^{10}$	&	1.09E-06	&	1.03 	&	1.09E-06	&	2.02 	&	1.41 	&	6.00 	\\
&	$2^{11}$	&	2.72E-07	&	8.79 	&	2.73E-07	&	2.00 	&	4.13 	&	6.00 	\\
&	$2^{12}$	&	6.80E-08	&	68.23 	&	6.81E-08	&	2.00 	&	36.87 	&	7.00 	\\
&	$2^{13}$	&	1.70E-08	&	540.58 	&	1.70E-08	&	2.00 	&	175.80 	&	8.00 	\\ \hline
\multirow{5}{*}{$(0.999,0.8)$}														
&	$2^{9}$	&	8.71E-06	&	0.05 	&	8.73E-06	&	-	&	0.50 	&	5.00 	\\
&	$2^{10}$	&	2.17E-06	&	1.01 	&	2.18E-06	&	2.00 	&	1.40 	&	6.00 	\\
&	$2^{11}$	&	5.44E-07	&	8.80 	&	5.44E-07	&	2.00 	&	4.64 	&	7.00 	\\
&	$2^{12}$	&	1.36E-07	&	68.54 	&	1.36E-07	&	2.00 	&	45.87 	&	9.00 	\\
&	$2^{13}$	&	3.40E-08	&	537.67 	&	3.40E-08	&	2.00 	&	227.41 	&	10.00 	\\ \hline
\end{tabular}
\end{table}

\begin{example}\label{ex2}
Consider the fractional convection equation (\ref{fde}) 
with $\alpha=0.5$, $\beta=1$, $D=1$, initial condition $u_{0}(x)=\exp(-x^2/\epsilon^2)/(\epsilon\sqrt{\pi})$ which tends to the Dirac delta function as $\epsilon\to 0$, source term $q(x, t)=0$, and exact solution the L\'{e}vy-Smirnov distribution $S(x)=\exp(-1/2x)/\sqrt{2\pi x^3}$ for $x>0$ and $S(x)=0$ for $x\leq 0$. 

In our computations, we take $\epsilon=2h$ in the initial condition $u_0(x)$ as an approximation to the Dirac delta function. Due to the long tail effect of distribution function, we need to take a large spatial interval such that the numerical solution at any time $t$ is close to zero on the boundary. Therefore we consider the computational domain $(x,t)\in[-1,20]\times[0,\sqrt{2}]$. On the other hand, we take $tol=10\tau^{3}$ to maintain second order convergence. The results are presented in Table~\ref{tab5}.
\end{example}

\begin{table}[H]
\captionsetup{font=footnotesize}
\caption{Results for Example~\ref{ex2} with the three mentioned approximations.}\label{tab5}
\centering
\fontsize{8}{8}\selectfont
\begin{tabular}{@{\extracolsep{5pt}}cccccccc}
\hline  & & \multicolumn{2}{c}{Direct method} & \multicolumn{4}{c}{GMRES method}\\     \cline{3-4} \cline{5-8}
 & $M$ & $E_{h^{2}}$ & Time & $E_{h^{2}}$ & Rate & Time & Iter\\
\hline
\multirow{5}{*}{Approx.~(\ref{2ndapp})}													
& $2^{11}$	&	3.59E-03	&	8.30 	&	3.59E-03	&	1.63 	&	2.35 	&	3.15 	\\
& $2^{12}$	&	1.01E-03	&	66.53 	&	1.01E-03	&	1.83 	&	11.99 	&	3.28 	\\
& $2^{13}$	&	2.61E-04	&	530.07 	&	2.61E-04	&	1.95 	&	40.53 	&	3.34 	\\
& $2^{14}$	&	-	&	-	&	6.55E-05	&	1.99 	&	152.84 	&	3.36 	\\
& $2^{15}$	&	-	&	-	&	1.64E-05	&	2.00 	&	566.98 	&	3.37 	\\ \hline
\multirow{5}{*}{Approx.~(\ref{ap1})}													
& $2^{11}$	&	2.36E-02	&	0.16 	&	2.36E-02	&	1.05 	&	1.64 	&	3.00 	\\
& $2^{12}$	&	9.84E-03	&	1.34 	&	9.84E-03	&	1.26 	&	4.23 	&	3.00 	\\
& $2^{13}$	&	3.54E-03	&	14.63 	&	3.54E-03	&	1.48 	&	12.04 	&	3.09 	\\
& $2^{14}$	&	1.16E-03	&	87.55 	&	1.16E-03	&	1.61 	&	21.32 	&	3.22 	\\
& $2^{15}$	&	3.78E-04	&	514.26 	&	3.78E-04	&	1.62 	&	95.20 	&	3.31 	\\ \hline
\multirow{5}{*}{Approx.~(\ref{ap2})}													
& $2^{11}$	&	2.31E-02	&	0.13 	&	2.31E-02	&	1.04 	&	1.68 	&	3.00 	\\
& $2^{12}$	&	9.82E-03	&	1.30 	&	9.82E-03	&	1.23 	&	4.14 	&	3.00 	\\
& $2^{13}$	&	3.61E-03	&	13.54 	&	3.61E-03	&	1.44 	&	11.93 	&	3.02 	\\
& $2^{14}$	&	1.20E-03	&	86.68 	&	1.20E-03	&	1.59 	&	22.35 	&	3.21 	\\
& $2^{15}$	&	3.86E-04	&	530.07 	&	3.86E-04	&	1.63 	&	100.98 	&	3.30 	\\ \hline
\end{tabular}
\end{table}

Table~\ref{tab5} illustrates that the schemes~(\ref{scheme}) and the equation settings provide convergent approximations for the distribution function. With the same number of temporal grid points $M$, longer CPU times but better accuracy and convergence orders are performed by using approximation~(\ref{2ndapp}) than by using approximations~(\ref{ap1}) and~(\ref{ap2}). For the GMRES solver, the iteration numbers are observed small, and the CPU times are smaller than that of the direct solver with increasing differences, demonstrating the efficiency of the proposed scheme with iterative solver. Besides, observing that the setting $\epsilon=2h=O(\tau^{2/(3-\alpha)})=O(\tau^{0.8})$ induces convergence of order about $1.6$ when using approximations~(\ref{ap1}) and~(\ref{ap2}), we remark here that the initial condition $u_{0}(x)$ probably introduces an error of $O(\epsilon^{2})$, dominating the order $3-\alpha$ convergence. Rigorous argument for this claim is to be studied.

\section{Conclusion}
In this paper, a second order scheme for a space fractional convection equation with order $\alpha\in(0,1)$ is first studied regarding to the generating function. By exploring the unconditional stability and consistency, two related approximations are proposed, with the resulting numerical schemes determined to be unconditionally stable and convergent with order $3-\alpha$ for a wide range of $\alpha$. Numerical results agree with the discussions.

\section{Acknowledgments}
This work was partially supported by the research grants MYRG2022-00262-FST, MYRG-GRG2023-00181-FST-UMDF from University of Macau;
and the Science and Technology Development Fund, Macau SAR under Funding Scheme for Postdoctoral Researchers of Higher Education Institutions 2021 (File No. 0032/APD/2021).


\end{document}